\documentclass[12pt,reqno]{amsart}
\usepackage[a4paper,margin=1in]{geometry}

\usepackage{amsmath, amsthm, amssymb, graphicx, stmaryrd}
\usepackage{extarrows}
\SetSymbolFont{stmry}{bold}{U}{stmry}{m}{n}
\usepackage[bookmarks=true, colorlinks, citecolor=blue, linkcolor=blue]{hyperref}
\usepackage{cleveref}
\usepackage{mathabx}
\usepackage{bm}
\usepackage{mathtools}  
\usepackage{extarrows}  

\usepackage[
  backend=biber,
  style=alphabetic,
  giveninits=true,
  maxbibnames=99,
  maxcitenames=99,
  maxalphanames=4,
  minalphanames=4
]{biblatex}

\newcommand{\eqviaII}[1]{\Longrightarrow{\text{(ii)}}}
\newcommand{\eqviaI}[1]{\xlongequal{\text{Step I}}}

\theoremstyle{plain}
\newtheorem{theorem}{Theorem}[section]
\newtheorem{lemma}[theorem]{Lemma}
\newtheorem{proposition}[theorem]{Proposition}

\crefname{theorem}{Theorem}{Theorems}
\crefname{definition}{Definition}{Definitions}
\crefname{lemma}{Lemma}{Lemmas}
\crefname{conjecture}{Conjecture}{Conjectures}

\theoremstyle{definition}
\newtheorem{definition}[theorem]{Definition}

\theoremstyle{remark}
\newtheorem{remark}[theorem]{Remark}

\newcommand{\bC}{\mathbb{C}}

\newcommand{\bQ}{\mathbb{Q}}
\newcommand{\bR}{\mathbb{R}}
\newcommand{\bZ}{\mathbb{Z}}

\newcommand{\bfd}{\mathbf{d}}

\newcommand{\cO}{\mathcal{O}}

\newcommand{\rme}{\mathrm{e}}

\newcommand{\Eff}{\mathrm{Eff}}
\newcommand{\ev}{\mathrm{ev}}

\title{Exponential Concentration for Quantum Periods via Mirror Symmetry}

\author{JIANXUN Hu}
\address{SCHOOL OF MATHEMATICS, SUN YAT-SEN UNIVERSITY, GUANGZHOU 510275, P.R. CHINA}
\email{stsjxhu@mail.sysu.edu.cn}

\author{HUA-ZHONG Ke}
\address{SCHOOL OF MATHEMATICS, SUN YAT-SEN UNIVERSITY, GUANGZHOU 510275, P.R. CHINA}
\email{kehuazh@mail.sysu.edu.cn}

\author{JINGWEI LU}
\address{SCHOOL OF MATHEMATICS, SUN YAT-SEN UNIVERSITY, GUANGZHOU 510275, P.R. CHINA}
\email{lujw28@mail2.sysu.edu.cn}

\date{\today}

\begin{document}

\begin{abstract}
We investigate power series satisfying the exponential concentration property, and show that suitable modifications of hypergeometric series respect this property. As a geometric application, we prove that the quantum period of a Fano manifold possesses the same property, whenever the manifold admits a convenient weak Landau–Ginzburg model with non-negative coefficients. 
\end{abstract}

\maketitle

\tableofcontents

\section{Introduction}

Let $X$ be a Fano manifold, i.e. a compact complex manifold with ample anti-canonical line bundle. Its quantum period $G_X(t)=\sum_{n\ge0}G_nt^n\in\bQ[[t]]$ is a power series with $G_n$ given by Gromov-Witten invariants \cite{CCGGK13}. The series distinguishes deformation families of Fano manifolds in dimension $\le3$ \cite{CCGK16}, and it is expected that this holds true in higher dimensions \cite{CK22}. The study of quantum periods sits at the crossroads of enumerative geometry, mirror symmetry and classification theory of Fano manifolds. It has attracted much attention in recent years.

In this article, we study the \emph{exponential concentration property} (see Definition \ref{def:exponentially-bounded}) for the summands of $G_X(t)=\sum_{n\ge0}G_nt^n$. Informally, this means that as $t\to+\infty$, the ratio of the sum of the dominant terms in $\{G_nt^n\}_{n\ge0}$ to the sum of the remaining terms grows exponentially in $t$. Our motivation comes from Gamma conjectures on the quantum cohomology of $X$ \cite{GGI2016}, which consist of Gamma conjecture I and II. Gamma I identifies a slowest-growing solution to the quantum differential equation of $X$ with the structure sheaf $\cO_X$ via the Gamma-integral structure \cite{Iritani2009IntegralStructure}. An equivalent formulation of Gamma I states that, assuming property $\cO$,
\begin{align}\label{eq-GammaIJoverG}
    \lim_{t\to+\infty}{J_X(t)\over G_X(t)}=\widehat{\Gamma}_X,
\end{align}
where $J_X(t)$ is Givental's $J$-function (an  $H^{\mathrm{ev}}(X)$-valued function), and $\widehat{\Gamma}_X\in H^{\mathrm{ev}}(X,\bR)$ is the Gamma class. We refer the reader to \cite{GGI2016,GolyshevZagier2016,GaIr2019ASPM,SandaShamoto2021,HKLY2021,Ke2024,Ch25} for progress on this conjecture, and \cite{Galkin2024RevisitingGammaI} for counterexamples and modifications. Recently, assuming (the original) Gamma I, J. Hu and H.-Z. Ke, together with C. Li and Z. Su, proposed a strategy to prove Gamma conjecture II, which identifies asymptotically exponential fundamental solutions to the quantum differential equation of $X$ with full exceptional collections in $\mathcal{D}^b(X)$. The strategy is applied to del Pezzo surfaces in a forthcoming article \cite{HKLSu26}. In a companion article \cite{HKLSo26}, the same strategy is carried out for Milnor hypersurfaces, with the additional input of mirror symmetry. Thus, beyond its intrinsic interest, Gamma I provides a pathway to  Gamma II, a conjecture that has attracted considerable attention (we refer the reader to Iritani's talk at ICM 2022 \cite{Iri23} and references therein for the status of Gamma II).

One difficulty in verifying Gamma I lies in evaluating the limit in \eqref{eq-GammaIJoverG}. Let $r$ be the Fano index of $X$, i.e., the largest positive integer such that $\frac{c_1(X)}{r}\in H^2(X,\bZ)$. Then we can write $J_X(t)=\mathrm{e}^{c_1(X)\log t}\sum_{n\ge0}J_{rn}t^{rn}$ and $G_X(t)=\sum_{n\ge0}G_{rn}t^{rn}$. In Theorem C of his paper \cite{Hugtenburg2024Quantum}, Hugtenburg showed that, assuming the ``superpolynomial peakedness" condition for $G_X(t)$---a condition that follows from exponential concentration---together with other hypotheses, the limit of the series-quotient in \eqref{eq-GammaIJoverG} equals the limit of the sequence-quotient
\begin{align}\label{eq-limequallim}
\lim_{t\to+\infty}\frac{J_X(t)}{G_X(t)}=\lim_{n\to\infty}\frac{\mathrm{e}^{c_1(X)\log\left(n\over C\right)}J_{rn}}{G_{rn}}\quad\text{($C$ depends on the concentration location)},
\end{align}
which simplifies the calculation of the LHS in \eqref{eq-GammaIJoverG}. He applied this result to prove Gamma I for a family of non-K\"ahler monotone symplectic manifolds. Implicitly, a similar method was used earlier by Golyshev--Zagier in the proof of Gamma I for Fano $3$-folds of Picard rank one \cite{GolyshevZagier2016}.

Inspired by the work of Golyshev--Zagier and Hugtenburg, we guess that the summands of quantum periods concentrate exponentially in general. The main result of this article is the following. 

\begin{theorem}\label{thm-verifyconj}(see also Theorem \ref{thm-mainthm})
    Let $X$ be a Fano manifold admitting a convenient weak Landau-Ginzburg model with non-negative coefficients. Then the summands of $G_X(t)=\sum_{n\ge0}G_nt^n$ concentrate exponentially near $n\approx T_{A,\mathrm{con}}t$.
\end{theorem}

In the above statement, $T_{A,\mathrm{con}}:=\varlimsup_{n\to\infty}|n!G_{n}|^{1\over n}$ denotes the \emph{$A$-model conifold value} of $X$. Emerging from discussions with Hugtenburg, the concept was introduced in \cite{Galkin2024RevisitingGammaI} in order to modify Gamma I, aiming to describe the growth rate of the solution to the quantum differential equation corresponding to $\cO_X$. It is expected that $T_{A,\mathrm{con}}$ is the critical value of the Lefschetz thimble corresponding to $\cO_X$ under homological mirror symmetry. 

According to predictions of mirror symmetry, $X$ should admit a Landau-Ginzburg model $(Y,f)$, such that the $A$-side of $X$ and the $B$-side of $(Y,f)$ are equivalent. Here $Y$ is a non-compact K\"{a}hler manifold and $f$ is a holomorphic function on $Y$. In Theorem \ref{thm-verifyconj}, we assume that $X$ admits an LG model in a ``weak'' sense. A \emph{weak} LG model of $X$ \cite{Prz11} is a pair $((\bC^\times)^m,f(\mathbf{x}))$ with $m=\dim_\bC X$ and $f(\mathbf{x})\in\bC[x_1^\pm,\dots,x^\pm_{m}]$, satisfying \[G_X(t)=1+\sum_{n\ge1}\frac{\mathrm{Cst}(f^n)}{n!}t^n,\] where $\mathrm{Cst}(f^n)$ is the constant term of the Laurent polynomial $f(\mathbf{x})^n$. We call the weak LG model \emph{convenient} if $f(\mathbf{x})$ is non-zero and its Newton polytope contains the origin in its interior. It is known that several classes of Fano manifolds admit convenient weak LG models with non-negative coefficients, e.g., Fano manifolds of dimension $\le3$, toric Fano manifolds, and partial flag manifolds.

Our proof of Theorem \ref{thm-verifyconj} relies on Galkin's random walk interpretation of $\mathrm{Cst}(f^n)$ \cite{Galkin2012SplitNotes}. Combined with the Local Central Limit Theorem, this interpretation implies that $G_X(t)$ can be viewed as a modified hypergeometric series. Theorem \ref{thm:intro-exp-bounded-GX-general} below establishes the exponential concentration property for a class of such modifications, generalizing  results from \cite[Section 3]{Paris2011Laplace} and \cite[Theorem D]{Hugtenburg2024Quantum}. This theorem is potentially of independent interest. Theorem \ref{thm-verifyconj} is then established by reducing it to the setting of Theorem \ref{thm:intro-exp-bounded-GX-general}.

\begin{theorem}\label{thm:intro-exp-bounded-GX-general}(see also Theorem \ref{thm-applicationmultigamma})
Let \(p\ge 0,q>0\) be integers. Let parameters
\begin{equation*}\label{eq:general-params}
\alpha_1,\ldots,\alpha_p,
\beta_1,\ldots,\beta_q,
\quad
a_1,\ldots,a_p,
b_1,\ldots,b_q,
\quad
T,
\end{equation*}
be positive numbers such that $\kappa \;:=\; \sum_{s=1}^{q}\beta_s \;-\; \sum_{r=1}^{p}\alpha_r\in\bZ_{>0}$. Write \(C \;:=\; \prod_{r=1}^{p}\alpha_r^{\alpha_r}\;\prod_{s=1}^{q}\beta_s^{-\beta_s}\). Let \(\{a_{n'}\}_{n'\ge0}\) be a sequence in $\bR_{>0}$ such that \(\frac{a_{n'+1}}{a_{n'}}=1+O(\frac{1}{n'})\) as $n'\to\infty$, and set
\begin{equation*}
H(x):=\sum_{n'=0}^{\infty}\frac{\displaystyle\prod_{r=1}^{p}\Gamma(\alpha_r n' + a_r)}
               {\displaystyle\prod_{s=1}^{q}\Gamma(\beta_s n' + b_s)}\;
               \,a_{n'}(Tx)^{\kappa n'}.
\end{equation*}
Then the summands of $H(x)=\sum_{n\ge0}H_nx^n$ concentrate exponentially near $n\approx \kappa C^{1\over\kappa}Tx$.
\end{theorem}

Besides exponential concentration, we also introduce the weaker notion of \emph{superpolynomial concentration} in this article (see Definition \ref{def:exponentially-bounded}). Informally, this means that as $t\to+\infty$, the ratio of the sum of the dominant terms of $\{G_nt^n\}_{n\ge0}$ to the sum of the remaining terms grows faster than any polynomial. Our motivation is to replace the cumbersome key hypothesis (``superpolynomially peaked") in Hugtenburg's Theorem C with a clean and transparent condition. When applied to the setting of Theorem C, superpolynomial concentration implies the key hypothesis used in the original proof, while all other assumptions remain in place. Consequently, this weaker condition alone suffices to establish the equality \eqref{eq-limequallim}. More broadly, our notions of ``exponential concentration" and ``superpolynomial concentration" are inspired by, but distinct from, Hugtenburg's notions of ``exponentially bounded" and ``superpolynomially peaked". The former are conceptually simpler, and they imply the defining properties of the latter (see Remark \ref{rmk-expconcentrationimpliesexpbound} and \ref{rmk-superpolyconcentrationimpliessuperpolypeak}).

\begin{remark}
   Hugtenburg expected the summands of $G_X(t)$ to concentrate near $n\approx Tt$ \cite[Remark 1.9]{Hugtenburg2024Quantum} in the sense of ``superpolynomially peaked", where $T$ is the spectral radius of quantum multiplication by $c_1(X)$. This is similar in spirit to Theorem \ref{thm-verifyconj}, but differs both in the precise concentration condition and in the location of concentration: our result establishes exponential concentration near $n\approx T_{A,\mathrm{con}}t$, and there are toric Fano manifolds with $T_{A,\mathrm{con}}\ne T$ \cite{Galkin2024RevisitingGammaI}. 
\end{remark}

This article is organized as follows. In Section 2, we study general properties of power series with exponential and superpolynomial concentration. In Section 3, we establish exponential concentration property for suitable modifications of hypergeometric series. In Section 4, we briefly review quantum periods and prove Theorem \ref{thm-verifyconj}.

 \textbf{Acknowledgements} 
The authors would like to thank Hiroshi Iritani, Kai Hugtenburg, Changzheng Li, Weiqiang He and Chaozhong Wu for helpful discussions. J. Hu and H.-Z. Ke are indebted to Sergey Galkin and Hiroshi Iritani for teaching them the random walk interpretation during earlier joint work, which is crucial in the proof of Theorem \ref{thm-verifyconj}. This work is supported in part by the National Key R \&  D Program of China No.~2023YFA1009801. J. Hu is also supported in part by NSFC Grant 12531003, and Science and Technology Planning Project of Guangdong Province(2025B1212140004). H.-Z. Ke is also supported in part by NSFC Grant 12271532.

\vspace{0.5cm}

\section{Exponential and superpolynomial concentration}
\label{subsec:exp-vs-superpoly}

In this section, we study general properties of power series with exponential concentration or superpolynomial concentration. The main results are Theorem \ref{lem:0} and Propositions \ref{lem:superpolynomial-peak-subpolynomial-closed}, \ref{prop-substitutionofg(n)tog(f(x))}.

For a power series $I(x)=\sum_{n\ge 0} a_n x^n\in\bR[[x]]$, we say that $I(x)$ is \emph{absolutely monotonic} if $I(x)$ converges for all $x>0$ and all $a_n$'s are non-negative.

\begin{definition}\label{def:exponentially-bounded}
Let $I(x)=\sum_{n\ge0}a_nx^n\ne0$ be absolutely monotonic, $f(x)$ a real polynomial with positive leading term, and $C,\nu,\alpha,\beta$ be positive numbers. Set $n_{\pm}(x):= \lfloor f(x)(1 \pm Cx^{-\nu}) \rfloor$ We say that \emph{the summands of $I(x)$ concentrate exponentially near $n\approx f(x)$ with window term $Cx^{-\nu}$ and parameters \((\alpha,\beta)\)}, if  as \(x\to+\infty\), we have
\begin{equation}\label{eq:exponentially-bounded}
\frac{1}{I(x)}\sum_{n=0}^{n_-(x)} a_n x^n = O\!\big(e^{-\alpha x^{\beta}}\big),
\quad
\frac{1}{I(x)}\sum_{n=n_+(x)}^{\infty} a_n x^n = O\!\big(e^{-\alpha x^{\beta}}\big);
\end{equation}we say that \emph{the summands of $I(x)$ concentrate superpolynomially near $n\approx f(x)$ with window term $Cx^{-\nu}$}, if for any $p\in\bR_{>0}$, as \(x\to+\infty\), we have
\begin{equation}\label{eq-defsuperpolynomialpeak}
\frac{1}{I(x)}\sum_{n=0}^{n_-(x)} a_n x^n=o(x^{-p}),\quad\frac{1}{I(x)}\sum_{n=n_+(x)}^{\infty} a_n x^n=o(x^{-p}).
\end{equation}
\end{definition}

\begin{lemma}\label{lemma-superpolynomialpeakparameters}
Let $I(x)=\sum_{n\ge0}a_nx^n\ne0$ be absolutely monotonic, and $f(x)$ a polynomial with positive leading term. Suppose that the summands of $I(x)$ concentrate superpolynomially near $n\approx f(x)$ with window term $Cx^{-\nu}$. Then:
\begin{enumerate}
\item[(i)] $f(x)$ is non-constant;
\item[(ii)] $\nu\le\deg f$;
\item[(iii)] there are infinitely many $n$'s with $a_n>0$
\end{enumerate}
\end{lemma}
\begin{proof}
We set $n_{\pm}(x):= \lfloor f(x)(1 \pm Cx^{-\nu}) \rfloor$.

For (i), we argue by contradiction and assume that $f(x)$ is a constant $c>0$. If $c$ is an integer, then for $x\gg 0$, we have $n_-(x)=c-1$ and $n_+(x)=c$, implying
\[
0=\lim_{x\to+\infty}\frac{1}{I(x)}\left(\sum_{n=0}^{c-1}a_nx^n+\sum_{n\ge c}a_nx^n\right)=1,\]
a contradiction. So $c$ is not an integer, and then for $x\gg 0$, we have $n_-(x)=n_+(x)=\lfloor c\rfloor$, implying
\[
0=\lim_{x\to+\infty}\frac{1}{I(x)}\left(\sum_{n=0}^{\lfloor c\rfloor}a_nx^n+\sum_{n\ge \lfloor c\rfloor}a_nx^n\right)\ge 1,
\]
a contradiction. This proves (i).

For (ii), we argue by contradiction and assume that $\nu>\deg f$. Let $\{x_m\}_{m\ge1}$ be a strictly increasing sequence tending to $+\infty$, such that $f(x_m)\in\bZ+{1\over2}$ for any $m$. Then $\lim_{m\to\infty}f(x_m)x_m^{-\nu}=0$ since $\nu>\deg f$. So $n_-(x_m)=n_+(x_m)=\lfloor f(x_m)\rfloor>0$ for $m>> 0$, implying
\[
0=\lim_{m\to\infty}\frac{1}{I(x_m)}\left(\sum_{n=0}^{\lfloor f(x_m)\rfloor}a_nx_m^n+\sum_{n\ge \lfloor f(x_m)\rfloor}a_nx_m^n\right)\ge1,
\]
a contradiction. This proves (ii).

For (iii), we argue by contradiction and assume that $I(x)$ is a polynomial. From (i) and (ii), we see that $n_-(x)\ge\deg I(x)$ when $x\gg 0$, implying $0=\lim_{x\to+\infty}\frac{1}{I(x)}\sum_{n=0}^{n_-(x)}a_nx^n=1$, a contradiction. This proves (iii).
\end{proof}

\begin{theorem}\label{lem:0}
Let $I(x)=\sum_{n\ge0}a_nx^n\ne0$ be absolutely monotonic, and $f(x)$ a polynomial of degree $d$ with positive leading coefficient $c_d$. Suppose that the summands of $I(x)$ concentrate superpolynomially near $n\approx f(x)$ with window term $Cx^{-\nu}$. For $x>0$, set $n_{\pm}(x):= \lfloor f(x)(1 \pm Cx^{-\nu}) \rfloor$ and \(\mu(x):=\max\{a_nx^n|n\ge0\}\).
Then:
\begin{enumerate}
\item[(i)] there exists \(X>0\) such that when \(x>X\), if $n\in\bZ_{\ge0}$ satisfies \(a_nx^n=\mu(x)\), then \(n\in (n_-(x),n_+(x))\);
\item[(ii)] 
\(
\lim_{x\to+\infty} \frac{\log \mu(x)}{x^d} = \lim_{x\to+\infty} \frac{\log I(x)}{x^d} = \frac{c_d}{d};
\)
\item[(iii)] \(\varlimsup_{n\to+\infty}\left(a_n\cdot\Gamma(\frac{n}{d}+1)\right)^{\frac{1}{n}}=\sqrt[d]{\frac{c_d}{d}}\).
\end{enumerate}
\end{theorem}

\begin{proof}
For (i), by Lemma \ref{lemma-superpolynomialpeakparameters} (i), there exists \(X_0>0\) such that, when $x>X_0$, we have
\begin{align*}
f(x)(1-\frac{3}{2}Cx^{-\nu})<n_-(x)&+1<n_+(x)\le f(x)(1+Cx^{-\nu})\\
\Rightarrow&\; n_+(x)-n_-(x)-1\le {5\over2}Cf(x)x^{-\nu}.
\end{align*}
By \eqref{eq-defsuperpolynomialpeak}, there exists \(X>X_0\) such that, for any $x>X$, we have
\[
\sum_{n_-(x)<n<n_+(x)}a_n x^n > \frac{1}{2}I(x)
\]
and, when $n\le n_-(x)$ or $n\ge n_+(x)$,
\[
a_n x^n < \frac{1}{5}\left(Cf(x)x^{-\nu}\right)^{-1}I(x).
\]
Using
\[
\mu(x)\cdot(n_+(x)-n_-(x)-1)\ge\sum_{n_-(x)<n<n_+(x)}a_nx^n,
\]
we get
\begin{align}\label{eq-inlemma0}
\mu(x)\ge\frac{\sum_{n_-(x)<n<n_+(x)}a_n x^n}{n_+(x)-n_-(x)-1}>\frac{1}{5}\left(Cf(x)x^{-\nu}\right)^{-1}I(x).
\end{align}
This proves (i). For $x>X$, let \(t(x):=\max\{n|a_nx^n=\mu(x),n_-(x)<n<n_+(x)\}\).

Now we prove (ii). Let $x_0>X$ be arbitrary. By continuity of $f(x)(1\pm Cx^{-\nu})$, there exists $\delta>0$ such that $x_0-\delta>X$ and for each $x\in(x_0-\delta,x_0+\delta)$, we have $n_-(x_0)-1\le n_-(x)< n_+(x)\le n_+(x_0)+1$, implying
\[
\mu(x)=\max\{a_nx^n|n_-(x_0)-1\le n\le n_+(x_0)+1\}.
\]
So $\mu(x)$ is continuous at $x_0$. Moreover, note that when $n_-(x_0)-1\le n\le n_+(x_0)+1$, we have $a_{t(x_0)}x_0^{t(x_0)}\ge a_nx_0^n$, and the equality holds only if $n\le t(x_0)$. This implies that there exists $\delta'\in(0,\delta)$ such that for any $x\in(x_0,x_0+\delta')$, we have $a_{t(x_0)}x^{t(x_0)}>a_nx^n$ when $n_-(x_0)-1\le n\le n_+(x_0)+1$ and $n\ne t(x_0)$, and for any $x\in(x_0-\delta',x_0)$, we have $a_{t(x_0)}x^{t(x_0)}>a_nx^n$ when $n_-(x_0)-1\le n\le n_+(x_0)+1$ and $a_{t(x_0)}x_0^{t(x_0)}\ne a_nx_0^n$. So $t(x)=t(x_0)$ for $x\in[x_0,x_0+\delta')$, and $t(x)$ is discontinuous at $x_0$ iff there exists $n< t(x_0)$ such that $a_{t(x_0)}x_0^{t(x_0)}=a_nx_0^n$. In sum, on $(X,+\infty)$, $\mu(x)$ is continuous, $t(x)$ is right continuous and piecewise constant, and $t(x)$ is discontinuous at $x_0>X$ only if there are $n\ne n'$ such that $a_nx_0^n=a_{n'}x_0^{n'}$.

Let us show that the set $S\subset(X,+\infty)$ of discontinuous points of $t(x)$ has no limit points. For any $B>X$, choose $B'$ large enough so that $n_+(x)=\lfloor f(x)(1+Cx^{-\nu})\rfloor<B'$ for all $x<B$. Then $t(x)$ is discontinuous at $x_0\in (X,B)$ only if there are $(n_1,n_2)\in B'\times B'$ with $n_1<n_2$ satisfying $a_{n_1}x_0^{n_1}=a_{n_2}x_0^{n_2}$; conversely, for each $(n_1,n_2)\in B'\times B'$ with $n_1<n_2$, there is at most one point $x_0\in(X,B)$ satisfying $a_{n_1}x_0^{n_1}=a_{n_2}x_0^{n_2}$. Note that there are only finitely many pairs $(n_1,n_2)\in B'\times B'$ with $n_1<n_2$, implying that $S\cap (X,B)$ is a finite set. This shows that $S$ has no limit points.

So there exists \(X<x_1<x_2<\dots\) and integer sequence $t_1,t_2,\dots$ such that we can write
\[
t(x) = t_k \text{ for } x\in[x_k,x_{k+1}),\; k=1,2,\dots.
\]
In each \([x_k,x_{k+1})\) we have \(\mu(x)=a_{t_k}x^{t_k}\), and therefore
\[
\frac{\mathrm{d}}{\mathrm{d}x}\log\mu(x) = \frac{\mathrm{d}}{\mathrm{d}x}\log(a_{t_k} x^{t_k}) = \frac{t_k}{x} = \frac{t(x)}{x},
\]
implying that
\[
\log\mu(z_2)-\log\mu(z_1) = \int_{z_1}^{z_2}\frac{t(s)}{s}\mathrm{d}s,\quad\forall z_1,z_2\in[x_k, x_{k+1}).\]
By continuity of $\mu(x)$, taking summation gives
\[
\log\mu(z_2)-\log\mu(z_1) = \int_{z_1}^{z_2}\frac{t(s)}{s}\mathrm{d}s,
\quad\forall z_1,z_2\ge x_1.
\]
By (i), for $s>X$, we have $n_-(s) < t(s) < n_+(s)$. Since \(d\ge1\) by Lemma \ref{lemma-superpolynomialpeakparameters} (i), it follows that as $s\to+\infty$, we have $n_{\pm}(s) \sim f(s) \sim c_d s^d$, implying $t(s) \sim c_d s^d$. Therefore, for any $\epsilon>0$, there exists $M_0>x_1$ such that for all $s>M_0$, $c_d(1-\epsilon) s^{d-1} < \frac{t(s)}{s} < c_d(1+\epsilon) s^{d-1}$. Then
\[
\int_{M_0}^x c_d(1-\epsilon) s^{d-1} \mathrm{d}s < \log \mu(x) - \log \mu(M_0) < \int_{M_0}^x c_d(1+\epsilon) s^{d-1} \mathrm{d}s,\quad\forall x>M_0.
\]
Note that
\(
\int_{M_0}^x c_d(1\pm\epsilon) s^{d-1} \mathrm{d}s=\frac{c_d}{d}(1\pm\epsilon)(x^d-M_0^d).
\)
As a consequence, there exists $M>M_0$ such that for any $x>M$, we have
\[
\frac{c_d}{d}(1 - 2\epsilon) x^d < \log \mu(x) - \log \mu(M_0) < \frac{c_d}{d}(1 + 2\epsilon) x^d.
\]
This proves (ii) for \(\mu(x)\).
For \(I(x)\), when \(x>X\), by \eqref{eq-inlemma0} we have \(I(x) < 5Cf(x)x^{-\nu}\mu(x)\).
Therefore, as $x\to+\infty$,
\[
\log \mu(x) < \log I(x) < \log \mu(x) + \log \left(5Cf(x)x^{-\nu}\right) = \log \mu(x) + O(\log x).
\]
This proves (ii).

It remains to prove (iii).
By (ii), for any \(\epsilon>0\), there exists \(X_\epsilon>0\) such that when \(x>X_\epsilon\), we have \(\log I(x)<\frac{c_d(1+\epsilon)}{d}x^d\).
Define \(x_n := (\frac{n}{c_d})^\frac{1}{d}\), i.e. \(n=c_d{x_n}^d\). Then there exists \(N_1>0\) such that, when \(n>N_1\), we have \(x_n>X_\epsilon\), implying \(\log I(x_n) < \frac{c_d(1+\epsilon)}{d}x_n^d = (1+\epsilon)\frac{n}{d}\), and consequently \(a_nx_n^n<I(x_n)\) gives
\[
\log a_n < \log I(x_n)-n\log x_n < (1+\epsilon)\frac{n}{d}-\frac{n}{d}\log\frac{n}{c_d}=\frac{n}{d}(1-\log \frac{n}{d})+\frac{n}{d}(\epsilon+\log \frac{c_d}{d}).
\]
Therefore, we have \(a_n < (\frac{n}{de})^{-\frac{n}{d}}\cdot(\sqrt[d]{e^\epsilon\frac{c_d}{d}})^n\) for any $n>N_1$, and by Stirling's formula, for any \(\epsilon'>\epsilon\), there exists \(N_2>N_1\) such that
\[
a_n < \frac{2\sqrt{2\pi\frac{n}{d}}}{\Gamma(\frac{n}{d}+1)}\cdot(\sqrt[d]{e^\epsilon\frac{c_d}{d}})^n < \frac{1}{\Gamma(\frac{n}{d}+1)}\cdot(\sqrt[d]{e^{\epsilon'} \frac{c_d}{d}})^n,\quad\forall n>N_2.
\]
Noting that \(0<\epsilon<\epsilon'\) can be freely chosen, this proves that
\[
\varlimsup_{n\to+\infty}\left(a_n\cdot\Gamma(\frac{n}{d}+1)\right)^{\frac{1}{n}}\le\sqrt[d]{\frac{c_d}{d}}.
\]
For the equality, we argue by contradiction and assume that there exists \(\epsilon''\in(0,\sqrt[d]{\frac{c_d}{d}})\) and \(N''>0\) such that \(\left(a_n\cdot\Gamma(\frac{n}{d}+1)\right)^{\frac{1}{n}}<\epsilon'',\;\forall n>N''\). Then \(I(x)\le H(x)+\sum_{n\ge0}\frac{1}{\Gamma(\frac{n}{d}+1)}(\epsilon''x)^n\) for any $x\ge0$, where \(H(x)\) is a polynomial. By the asymptotics of hypergeometric series (see \cite[Section 6]{Paris2011Laplace}), as \(x\to+\infty\), we have
\[
\sum_{n\ge0}\frac{1}{\Gamma(\frac{n}{d}+1)}(\epsilon''x)^n\sim de^{(\epsilon''x)^d},
\]
implying
\[
\log I(x)<\log\left(2\sum_{n\ge0}\frac{1}{\Gamma(\frac{n}{d}+1)}(\epsilon''x)^n\right)\sim\epsilon''^d x^d,
\]
contradicting (ii). This proves (iii).
\end{proof}

\begin{lemma}\label{lemma-superpolynomialpeakparameterchange}
Let $I(x)=\sum_{n\ge0}a_nx^n\ne0$ be absolutely monotonic, and $f(x)$ (resp. $g(x)$) a polynomial of degree $d_f$ (resp. $d_g$) with positive leading coefficient $c_f$ (resp. $c_g$). Suppose that the summands of $I(x)$ concentrate superpolynomially near $n\approx f(x)$ with window term $Cx^{-\nu}$. Then the following hold.

\begin{enumerate}
\item[(i)]
For any $C'\ge C$ and $\nu'\in(0,\nu]$, the summands of $I(x)$ concentrate superpolynomially near
$n\approx f(x)$ with window exponent $C'x^{-\nu'}$.

\item[(ii)]
The summands of $I(x)$ concentrate superpolynomially near $n\approx g(x)$ (possibly with different window terms) if and only if $d_f=d_g$ and $c_f=c_g$.
\end{enumerate}
\end{lemma}

\begin{proof}
For $C'>0,\,\nu'>0$, and a polynomial $h(x)$, we set 
\[
n^{h,C',\nu'}_{\pm}(x):=\bigl\lfloor h(x)(1\pm C'x^{-\nu'})\bigr\rfloor.
\]

For (i), since $C'\ge C$ and $\nu'\in(0,\nu]$, it follows that when $x\gg0$, we have
\[
n^{f,C',\nu'}_-(x)\le n^{f,C,\nu}_-(x),
\quad
n^{f,C,\nu}_+(x)\le n^{f,C',\nu'}_+(x),
\]
implying
\[
\sum_{n=0}^{\,n^{f,C',\nu'}_-(x)} a_nx^n
\le
\sum_{n=0}^{\,n^{f,C,\nu}_-(x)} a_nx^n,
\quad
\sum_{n=n^{f,C',\nu'}_+(x)}^\infty a_nx^n
\le
\sum_{n=n^{f,C,\nu}_+(x)}^\infty a_nx^n.
\]
This proves (i).

For (ii), the 'only if' part comes from Theorem \ref{lem:0} (ii). For the 'if' part, note that $f(x)-g(x)$ is a polynomial with degree$< d_f=d_g$. Let $\nu'\in(0,\nu)$ and $C'>0$. When $x\gg0$, we have \(|f(x)-g(x)|+Cf(x)x^{-\nu}\le C'g(x)x^{-\nu'}\), implying
\[
Cf(x)x^{-\nu}-C'g(x)x^{-\nu'}\le f(x)-g(x)\le C'g(x)x^{-\nu'}-Cf(x)x^{-\nu}.
\]
Now the left inequality gives $n^{g,\nu',C'}_-(x)\le n^{f,\nu,C}_-(x)$, and the right inequality gives $n^{f,\nu,C}_+(x)\le n^{g,\nu',C'}_+(x)$. This proves (ii).
\end{proof}

\begin{remark}\label{remark-monomial-peak-reduction}
By Lemma \ref{lemma-superpolynomialpeakparameterchange} (ii), the summands of $I(x)$ concentrate superpolynomially near $n\approx c_fx^{d_f}$.
\end{remark}

\begin{definition}\label{def:subpolynomial}
A sequence $\{b_n\}_{n\ge0}$ in $\bR$ is called \emph{subpolynomial} if
$\displaystyle \lim_{n\to\infty}\frac{b_n}{n^p}=0$ for every $p\in\bR_{>0}$. 
\end{definition}

\begin{remark}\label{remark-boundforsubpolynomialseq}
An equivalent characterization for $\{b_n\}$ being subpolynomial is as follows: for each $p>0$, there exists $A_p>0\,,B_p>0$ such that $|b_n|\le A_p+B_pn^p$ for all $n$. 
\end{remark}

\begin{proposition}\label{lem:superpolynomial-peak-subpolynomial-closed}
Let $I(x)=\sum_{n\ge0}a_nx^n\ne0$ be absolutely monotonic, $f(x)$ a polynomial of degree $d$ with positive leading term $c_d$, and $C,\nu,\alpha,\beta$ be positive numbers. Set $n_{\pm}(x):= \lfloor f(x)(1 \pm Cx^{-\nu}) \rfloor$. Then the following hold.
\begin{itemize}
    \item[(i)] If the summands of $I(x)$ concentrate superpolynomially near $n\approx f(x)$ with window term $Cx^{-\nu}$, then for any subpolynomial sequence $\{b_n\}_{n\ge0}$ and any \(p>0\), as $x\to+\infty$ we have
\[
\frac{1}{I(x)}\sum_{n=0}^{n_-(x)} b_n a_n x^n=o(x^{-p}),\quad\frac{1}{I(x)}\sum_{n=n_+(x)}^{\infty} b_n a_n x^n=o(x^{-p}).
\]
\item[(ii)] If the summands of $I(x)$ concentrate exponentially near $n\approx f(x)$ with window term $Cx^{-\nu}$ and parameters \((\alpha,\beta)\), then for any subpolynomial sequence $\{b_n\}_{n\ge0}$ and any \(\alpha'\in(0,\alpha)\), as $x\to+\infty$ we have
\[
\frac{1}{I(x)}\sum_{n=0}^{n_-(x)} b_n a_n x^n=O\!\big(e^{-\alpha' x^{\beta}}\big),\quad\frac{1}{I(x)}\sum_{n=n_+(x)}^{\infty} b_n a_n x^n=O\!\big(e^{-\alpha' x^{\beta}}\big).
\]
\end{itemize}
\end{proposition}

\begin{proof}
For (i), we fix $q>0$. It follows from Remarks \ref{remark-boundforsubpolynomialseq} that $|b_n|\le A_q+B_qn^q$ for all $n$. Therefore,
\[
\frac{|\sum_{n\le n_-(x)}b_n a_n x^n|}{I(x)}\le\left(A_{q}+B_{q}\,(f(x)(1-Cx^{-\nu}))^{q}\right)\frac{\sum_{n\le n_-(x)}a_n x^n}{I(x)}.
\]
This proves the first equality by definition of superpolynomial concentration.
For the second equality, similar argument shows that for any \(p>0\), we have
\(
\frac{\sum_{n_+(x)\le n\le x^{2d}}b_n a_n x^n}{I(x)}=o(x^{-p})\) as \(x\to+\infty\). For \(n>x^{2d}\), we have
\[
|\sum_{n>x^{2d}}b_n a_n x^n |\le \sum_{n>x^{2d}}(A_q+B_q n^q) 2^{-n} a_n (2x)^n\le I(2x)\sum_{n>x^{2d}}(A_q+B_q n^q) 2^{-n}.
\]
As $x\to+\infty$, Theorem \ref{lem:0} (ii) gives $\frac{I(2x)}{I(x)}<\exp{\left(\frac{3}{2} \frac{c_d}{d}(2x)^d-\frac{1}{2}\frac{c_d}{d}x^d\right)}$, and the following Lemma \ref{lem:K/2} gives $ \sum_{n>x^{2d}}(A_q+B_q n^q) 2^{-n}<2^{-x^{2d}/2}$, which imply that
\[
\frac{|\sum_{n>x^{2d}}b_n a_n x^n|}{I(x)}<\exp{\left((2^d \frac{3c_d}{2d}-\frac{c_d}{2d})x^d\right)}\cdot 2^{-x^{2d}/2}.
\]
Consequently, for any $p>0$, we have 
\(
\frac{|\sum_{n>x^{2d}}b_n a_n x^n|}{I(x)}=o(x^{-p})\) as \(x\to+\infty\). This proves (i).

Arguments for (ii) are similar. Note that the range \(n>x^{2d}\) can be replaced by any \(n>x^{Md}\) with some constant \(M>0\), to fit parameters \(\alpha,\beta\).
\end{proof}

\begin{lemma}\label{lem:K/2}
Let \(h(x)\) be a polynomial with positive leading term. Then for \(K\gg 0\), we have
\(\sum_{n>K}h(n) 2^{-n}<2^{-K/2}\).
\end{lemma}

\begin{proof}
Let \(c_n := h(n) 2^{-n}\), and then \(\lim_{n\to\infty}\frac{c_{n+1}}{c_n}=\frac{1}{2}\).
Note that $h(x)$ has positive leading term. So there exists \(N>0\) such that for \(n>N\), we have $c_n>0$ and \(\frac{c_{n+1}}{c_n}<\frac{3}{4}\). Now for \(K>N\), we have
\(\sum_{n>K}c_n < c_{K+1}[1+\frac{3}{4}+\left(\frac{3}{4}\right)^2+...] = 4c_{K+1} = 4h(K+1) 2^{-(K+1)}\). Note that \(4h(K+1) 2^{-(K+1)}<2^{-K/2}\) when $K\gg 0$.
This finishes the proof.
\end{proof}

\begin{remark}\label{rmk-expconcentrationimpliesexpbound}
    From Proposition \ref{lem:superpolynomial-peak-subpolynomial-closed} (ii), one can see that ``exponential concentration" implies ``exponentially bounded" \cite[Definition 4.1]{Hugtenburg2024Quantum}.
\end{remark}

\begin{proposition}\label{prop-substitutionofg(n)tog(f(x))}
Let $I(x)=\sum_{n\ge0}a_nx^n\ne0$ be absolutely monotonic, and $f(x)$ a polynomial with positive leading term, such that the summands of $I(x)$ concentrate superpolynomially near $n\approx f(x)$ with window term $Cx^{-\nu}$. For $N\ge0$, let $g(y)$ be a differentiable function on $[N,+\infty)$ such that for any $p\in\bR_{>0}$, we have $yg'(y)=o(y^p)$ as $y\to+\infty$. Then for any subpolynomial sequence $\{b_n\}_{n\ge0}$, we have
\[
\lim_{x\to+\infty}\frac{\sum_{n\ge N}a_nb_ng(n)x^n-g(f(x))\sum_{n\ge N}a_nb_nx^n}{I(x)}=0.
\]
\end{proposition}
\begin{proof}
By Lemma \ref{lemma-superpolynomialpeakparameterchange} (i), we can assume $\nu\in(0,1)$. Set $n_\pm:=\lfloor f(x)(1\pm Cx^{-\nu})\rfloor$. When $x\gg 0$, we have $N<n_-(x)<n_+(x)$, and then
\begin{align*}
&\frac{|\sum_{n \ge N} a_n b_n g(n) x^n-g(f(x)) \sum_{n \ge N} a_n b_n x^n|}{I(x)}\\
\le{}&\frac{\sum_{n_-(x)<n<n_+(x)}a_n|b_n|x^n\cdot|g(n)-g(f(x))|}{I(x)}\\
+\Bigg(&\frac{\sum_{n=\lceil N\rceil}^{n_-(x)}a_n\cdot|b_ng(n)|x^n}{I(x)}+|g(f(x))|\frac{\sum_{n=\lceil N\rceil}^{n_-(x)}a_n\cdot |b_n|x^n}{I(x)}\\
+\;\;\,&\frac{\sum_{n\ge n_+(x)}a_n\cdot|b_ng(n)|x^n}{I(x)}+|g(f(x))|\frac{\sum_{n\ge n_+(x)}a_n\cdot|b_n|x^n}{I(x)}\Bigg)\\
=:{}&M_0(x)+\left(R_1(x)+R_2(x)+R_3(x)+R_4(x)\right).
\end{align*}
Note that the assumptions on $g'(y)$ implies that $\lim_{y\to+\infty}\frac{g(y)}{y^p}=0$ for any $p>0$. So both $\{|b_n|\}$ and $\{|b_ng(n)|\}$ are subpolynomial, and it follows from Proposition \ref{lem:superpolynomial-peak-subpolynomial-closed} that for any $p>0$, $R_i(x)=o(x^{-p})$ as $x\to+\infty$. It remains to show $\lim_{x\to+\infty}M_0(x)=0$.

When $x\gg0$, for $n_-(x)<n<n_+(x)$, we have
\[
|g(n)-g(f(x))|\le\max\{M(x),m(x)\},
\]
where
\begin{align*}
M(x)&:=\max\{|g(n)-g(f(x))|:f(x)\le n\le f(x)(1+Cx^{-\nu})\},\\
m(x)&:=\max\{|g(n)-g(f(x))|:f(x)(1-Cx^{-\nu})-1\le n\le f(x)\}.
\end{align*}
Let $d:=\deg f(x)\ge1$. Note that $yg'(y)=o(y^{\frac{\nu}{2d}})$ as $y\to+\infty$, and thus there exists $N'>N$ and $M>0$ such that $|g'(y)|\le My^{\frac{\nu}{2d}-1}$ for all $y>N'$. So when $x$ is large enough such that $n_-(x)>N'$, we have
\begin{align*}
M(x)&\le \max\{|g'(y)|:f(x)\le y\le f(x)(1+Cx^{-\nu})\}\cdot Cf(x)x^{-\nu}\\
&\le MCf(x)^{\frac{\nu}{2d}-1}f(x)x^{-\nu}=MCf(x)^{\frac{\nu}{2d}}x^{-\nu},\\
m(x)&\le\max\{|g'(y)|:f(x)(1-Cx^{-\nu})-1\le y\le f(x)\}\cdot (Cf(x)x^{-\nu}+1)\\
&\le M(f(x)(1-Cx^{-\nu})-1)^{\frac{\nu}{2d}-1}(Cf(x)x^{-\nu}+1)\\
&= MCf(x)^{\frac{\nu}{2d}}\frac{x^{-\nu}+(Cf(x))^{-1}}{(1-Cx^{-\nu}-f(x)^{-1})^{1-\frac{\nu}{2d}}}.
\end{align*}
From $\nu\in(0,1)$, we see that both $M(x)$ and $m(x)$ are $O(x^{-\frac{\nu}{2}})$ as $x\to+\infty$,  and then
\[
M_0(x)\le\frac{\sum_{n_-(x)<n<n_+(x)}a_n|b_n|x^n}{I(x)}\cdot O(x^{-\frac{\nu}{2}}).
\]
Since $\{b_n\}_{n\ge1}$ is subpolynomial, it follows from Remark \ref{remark-boundforsubpolynomialseq} that, for each $q>0$, there exists positive numbers $A_q,B_q$ such that $|b_n|\le A_q+B_qn^q$ for all $n$, Consequently, as $x\to+\infty$, we get
\begin{align*}
&\frac{\sum_{n_-(x)<n<n_+(x)}a_n|b_n|x^n}{I(x)}\cdot O(x^{-\frac{\nu}{2}})\\
\le &\frac{\sum_{n_-(x)<n<n_+(x)}a_nx^n}{I(x)}(A_q+B_q\cdot n_+(x)^{q})\cdot O(x^{-\frac{\nu}{2}})\\
\le &\left(A_q+B_q\cdot \left(f(x)(1+Cx^{-\nu})\right)^{q}\right)\cdot O(x^{-\frac{\nu}{2}}).
\end{align*}
Choosing $q=\frac{\nu}{4d}$, we see that $\lim_{x\to+\infty}M_0(x)=0$. This proves the proposition.
\end{proof}

\begin{remark}\label{rmk-superpolyconcentrationimpliessuperpolypeak}
    Taking $g(y)=(\log y)^k$ with $k\in\bZ_{\ge0}$ in Proposition \ref{prop-substitutionofg(n)tog(f(x))}, one can see that ``superpolynomial concentration" implies ``superpolynomially peaked" \cite[Definition 1.8]{Hugtenburg2024Quantum}.
\end{remark}

\vspace{1cm}

\section{Exponential concentration for modified hypergeometric series}\label{section: verify exp bound}

In this section, we prove Theorem \ref{thm-applicationmultigamma}, which states that certain modification of hypergeometric series satisfies the exponential concentration property.

\begin{lemma}\label{lem:An-basic-properties}
Let \(\{A_n\}_{n\ge0}\) be a sequence in $\bR_{>0}$ such that \(\frac{A_{n+1}}{A_n}=1+O\!\left(\frac1n\right)\) as \(n\to\infty\). Then the following four statements hold.
\begin{enumerate}
\item[(i)]We have \(\log A_n=O(\log n)\) as \(n\to\infty\).
\item[(ii)] Suppose that \(\{M_n\}_{n\ge1}\) is a sequence in $\bR_{>0}$ such that \(M_n=o(n)\) as $n\to\infty$. Then as $n\to\infty$, we have \(\max_{|m|\le M_n}\left|\log\frac{A_{n+m}}{A_n}\right|=O\!\left(\frac{M_n}{n}\right)\) and \(\max_{|m|\le M_n}\left|\frac{A_{n+m}}{A_n}-1\right|=O\!\left(\frac{M_n}{n}\right)\).
\item[(iii)] For each \(\theta\in(0,1)\), we have \(\lim_{n\to\infty}\max_{|m|\le n^\theta}\left|\frac{A_{n+m}}{A_n}-1\right|=0\).
\item[(iv)] There exists \(M>0\) such that for all integers \(m,n\ge0\),
\[
\left(\frac{m+1}{n+1}\right)^{-M}
\le
\frac{A_m}{A_n}
\le
\left(\frac{m+1}{n+1}\right)^M.
\]
\end{enumerate}
\end{lemma}

\begin{proof}
Note that there exists \(C\in\bR_{>0}\) and \(N_0\in\bZ_{>0}\) such that for any $n\ge N_0$, we have \(
\left|\frac{A_{n+1}}{A_n}-1\right|\le \frac{C}{n}\le\frac{1}{2}
\), and then, using $|\log(1+x)|\le2|x|$ for $x>-\frac{1}{2}$, we get $\left|\log\frac{A_{n+1}}{A_n}\right|\le \frac{2C}{n}$.

For (i), when $n\ge N_0$, we have
\[
|\log A_n|
\le
|\log A_{N_0}|+\sum_{k=N_0}^{n-1}\left|\log\frac{A_{k+1}}{A_k}\right|\le|\log A_{N_0}|+2C\sum_{k=N_0}^{n-1}\frac{1}{k}.
\]
Noting $\sum_{k=1}^n\frac{1}{k}=O(\log n)$  as  $n\to\infty$, this proves (i).

For (ii), note that there exists $N_1\in\bZ_{>0}$ such that, for any $n\ge N_1$, we have $M_n\le\frac{n}{2}$. So when $n\ge N_1$, if $m\in\bZ$ satisfies $|m|\le M_n$, then $m+n>0$, implying that $\log\frac{A_{n+m}}{A_n}$ is well-defined. Now suppose that $n\ge\max\{2N_0,N_1\}$. Let $m\in\bZ$ satisfy $|m|\le M_n$. If $m\ge0$, then
\[
\left|\log\frac{A_{n+m}}{A_n}\right|
\le\sum_{k=n}^{n+m-1}\left|\log\frac{A_{k+1}}{A_k}\right|\le 2C\sum_{k=n}^{n+m-1}\frac1k\le 2C\frac{m}{n}\le 2C\frac{M_n}{n};
\]
if \(m=-\ell<0\), then \(0\le \ell\le M_n\le\frac{n}{2}\), implying $n-\ell\ge\frac{n}{2}\ge N_0$, and as a consequence,
\[
\left|\log\frac{A_{n-\ell}}{A_n}\right|
\le
\sum_{k=n-\ell}^{n-1}\left|\log\frac{A_{k+1}}{A_k}\right|\le 2C\sum_{k=n-\ell}^{n-1}\frac1k\le2C\frac{\ell}{n-\ell}\le2C\frac{M_n}{\frac{n}{2}}.
\]
So we obtain \(\max_{|m|\le M_n}\left|\log\frac{A_{n+m}}{A_n}\right|\le 4C\frac{M_n}{n}\), implying the first equality. Noting $|\log(1+x)|\ge\frac{|x|}{2}$ for $x\in(-1,1)$, the second equality is a consequence of the first one. This proves (ii).

For (iii), the required equality follows from (ii) by letting $M_n:=n^\theta$. This proves (iii).

For (iv), first note that there exists $C'>C$ such that \(\left|\log\frac{A_{k+1}}{A_k}\right|\le \frac{C'}{k+1}\) for all $k\in\bZ_{\ge0}$. We set $M:=C'$. For \(m\ge n\ge0\), we have
\[
\left|\log\frac{A_m}{A_n}\right|
\le
\sum_{k=n}^{m-1}\left|\log\frac{A_{k+1}}{A_k}\right|
\le
M\sum_{k=n}^{m-1}\frac1{k+1}
\le M\sum_{k=n}^{m-1}\int_{k}^{k+1}\frac{\mathrm{d}x}{x}\le
M\log\frac{m+1}{n+1},
\]
implying the required inequalities. The case \(n>m\ge 0\) is similar. This proves (iv).
\end{proof}

\begin{theorem}\label{thm:package-Pfw-exp-peak}
Suppose that \(f(x)\) is a polynomial of degree \(d\ge 1\) with positive leading coefficient $c_d x^d$. Let \(\{W_n(x)\}_{n\ge0}\) be a sequence of nonnegative functions on some interval \((X_0,\infty)\), and let \(\{A_n\}_{n\ge0}\) be a sequence in $\bR_{>0}$ such that
\(\frac{A_{n+1}}{A_n}=1+O\!\left(\frac1n\right)\) as \(n\to\infty\).
Fix \(\nu\in(0,d), C\in\bR_{>0}\), and set
\[
N(x):=\lfloor f(x)\rfloor,\qquad
n_\pm(x):=\bigl\lfloor f(x)(1\pm Cx^{-\nu})\bigr\rfloor .
\]
Assume that there exist positive constants \(c_0,D,K\) and \(q\in(0,1),\gamma\in(0,d-\nu)\), such that, when $x$ is large enough, the following conditions (i) and (ii) hold:

\smallskip
\noindent
(i) for every \(0\le i\le N(x)\) and \(0\le j\le D N(x)\),
\[
W_{N(x)-i}(x)\le KW_{N(x)}(x)\exp\!\left(-c_0\frac{i^2}{x^{2\gamma}}\right),
\;
W_{N(x)+j}(x)\le KW_{N(x)}(x)\exp\!\left(-c_0\frac{j^2}{x^{2\gamma}}\right);
\]
\smallskip
\noindent
(ii) for every \(j\ge \lfloor D N(x)\rfloor\),
\[
W_{N(x)+j+1}(x)\le q\cdot W_{N(x)+j}(x).
\]

\noindent Define \(H(x):=\sum_{n=0}^\infty A_n W_n(x)\). Then for any \(\alpha\in(0,c_0C^2c_d^2)\), as $x\to+\infty$,
\[
\frac{\sum_{n=0}^{n_-(x)}A_n W_n(x)}{ H(x)}
=
O\!\bigl(e^{-\alpha x^{2(d-\nu-\gamma)}}\bigr),
\quad
\frac{\sum_{n=n_+(x)}^\infty A_nW_n(x)}{ H(x)}
=
O\!\bigl(e^{-\alpha x^{2(d-\nu-\gamma)}}\bigr).
\]
\end{theorem}

\begin{proof}
Write \(L(x):=Cf(x)x^{-\nu}\). Then \(L(x)\asymp x^{d-\nu}\) as $x\to+\infty$. For all large \(x\), we have
\[
n\le n_-(x)\ \text{or}\ n\ge n_+(x)\quad\Longrightarrow\quad |n-N(x)|\ge \tfrac12L(x).
\]
By Lemma \ref{lem:An-basic-properties}(iv), there exists $M>0$ such that
\[
A_n\le A_{N(x)}\max\left\{\left(\frac{n+1}{N(x)+1}\right)^M,\left(\frac{N(x)+1}{n+1}\right)^M\right\},
\quad
\forall n\ge0.
\]
Hence there exists $D'>0$ such that, when $x$ is large enough, we have
\[
0\le n\le (1+D)N(x)\Longrightarrow A_n<D'\cdot A_{N(x)} x^{Md}.
\]
Note that as $x\to+\infty$, \(|n-N(x)|\ge \frac{L(x)}{2}\) implies \(\frac{(n-N(x))^2}{x^{2\gamma}}\ge \frac{L(x)^2}{4x^{2\gamma}}\asymp x^{2(d-\nu-\gamma)}\). Consequently, fix \(\varepsilon\in(0,\frac{c_0}{2})\), and then for large $x$, we have
\begin{equation}\label{eq:weighted-absorb}
0\le n\le (1+D)N(x),\ |n-N(x)|\ge \frac{L(x)}{2}\Longrightarrow A_n\le A_{N(x)}\exp\!\left(\varepsilon\frac{(n-N(x))^2}{x^{2\gamma}}\right).
\end{equation}
For any $S\in(0,1)$, when $x$ is large enough, we have $N(x)+\lfloor DN(x)\rfloor\ge n_+(x)$, $N(x)-n_-(x)\ge SL(x)$ and $n_+(x)-N(x)\ge SL(x)$. Therefore, using condition (i), we have
\begin{align*}
&\sum_{n=0}^{n_-(x)}A_n W_n(x)
+
\sum_{n=n_+(x)}^{N(x)+\lfloor DN(x)\rfloor}A_n W_n(x)\\
\le{}&
KA_{N(x)}W_{N(x)}
\sum_{\substack{0\le n\le N(x)+\lfloor DN(x)\rfloor\\ |n-N(x)|\ge SL(x)}}
\exp\!\left(-(c_0-\varepsilon)\frac{(n-N(x))^2}{x^{2\gamma}}\right)\\
\le{}&
KA_{N(x)}W_{N(x)}\cdot(N(x)+\lfloor DN(x)\rfloor)
\exp\!\left(-(c_0-\varepsilon)S^2C^2c_d^2x^{2(d-\nu-\gamma)}\right)\\
=&
O\!\left(A_{N(x)}W_{N(x)} e^{-\alpha_1 x^{2(d-\nu-\gamma)}}\right).
\end{align*}
By choosing $\epsilon$ and $1-S$ small enough, we can take $\alpha_1$ to be any number in $(0,c_0C^2c_d^2)$. Using \(H(x)\ge A_{N(x)}W_{N(x)}\), this proves the first required equality.

Note that \(A_{n+1}/A_n=1+O(1/n)\). So there exists \(q_1\in(q,1)\) such that for all large \(x\) ,
\[
n\ge N(x)+\lfloor DN(x)\rfloor\Longrightarrow\frac{A_{n+1}}{A_n}\le \frac{q_1}{q}
\; \xLongrightarrow{\text{condition (ii)}} A_{n+1}W_{n+1}(x) \le q_1 A_n W_n(x),
\]
and therefore,
\[
\sum_{n\ge N(x)+\lfloor DN(x)\rfloor}A_n W_n(x)
\le A_{N(x)+\lfloor DN(x)\rfloor} W_{N(x)+\lfloor DN(x)\rfloor}(x)\cdot\sum_{i=0}^\infty q_1^i.
\]
As $x\to+\infty$, applying \eqref{eq:weighted-absorb} at \(n=N(x)+\lfloor DN(x)\rfloor\) and using condition (i), we have
\begin{align*}
A_{N(x)+\lfloor DN(x)\rfloor} W_{N(x)+\lfloor DN(x)\rfloor}(x)
\le{}&
KA_{N(x)}W_{N(x)}
\exp\!\left(-(c_0-\varepsilon)\frac{D^2N(x)^2}{4x^{2\gamma}}\right)\\
=&
O\!\left(A_{N(x)}W_{N(x)} e^{-\alpha_2 x^{2(d-\gamma)}}\right)
\end{align*}
for some \(\alpha_2>0\), and as a consequence,
\[
\sum_{n\ge N(x)+\lfloor DN(x)\rfloor}A_n W_n(x)
=
O\!\left(A_{N(x)}W_{N(x)} e^{-\alpha_3 x^{2(d-\nu-\gamma)}}\right), \quad\forall\alpha_3>0.
\]
Using \(H(x)\ge A_{N(x)}W_{N(x)}\), this proves the second required equality. This completes the proof.
\end{proof}

\begin{theorem}\label{thm:exp-peak-and-asymptotic}
Let \(\{A_{n'}\}_{n'\ge0}\) be a sequence in $\bR_{>0}$ such that \(\frac{A_{n'+1}}{A_{n'}}=1+O\!\left(\frac{1}{n'}\right)\) as \(n'\to\infty\). Let \(\kappa,T\) be positive numbers, and define
\[
W_{n'}(x):=\frac{(Tx)^{\kappa n'}}{\Gamma(\kappa n'+1)},
\qquad
H(x):=\sum_{n'=0}^\infty A_{n'} W_{n'}(x).
\]
Fix \(\nu\in(0,\tfrac12)\), and set \(n'_\pm(x):=\Bigl\lfloor \frac{Tx}{\kappa}\bigl(1\pm (Tx)^{-\nu}\bigr)\Bigr\rfloor \).Then there exists \(\alpha>0\) such that the following two statements hold:
\begin{itemize}
\item[(i)] as \(x\to+\infty\), we have
\[
\frac{\sum_{n'=0}^{n'_-(x)}A_{n'} W_{n'}(x)}{H(x)}
=
O\!\bigl(e^{-\alpha x^{1-2\nu}}\bigr),
\qquad
\frac{\sum_{n'=n'_+(x)}^\infty A_{n'} W_{n'}(x)}{H(x)}
=
O\!\bigl(e^{-\alpha x^{1-2\nu}}\bigr);
\]
\item[(ii)] when $\kappa\in\bZ_{>0}$, the summands of $H(x)=\sum_{n\ge0}H_nx^n$ concentrate exponentially near $n\approx Tx$, with window term $T^{-\nu}x^{-\nu}$ and parameters $(\alpha,1-2\nu)$.
\end{itemize}
\end{theorem}

\begin{proof}
For (i), we apply Theorem \ref{thm:package-Pfw-exp-peak} with
\[
f(x)=\frac{Tx}{\kappa},\quad C=T^{-\nu},\quad d=1,\quad \gamma=1/2.
\]
This force $\nu$ to be in $(0,\frac{1}{2})$, since $\gamma<d-\nu$. It remains to verify hypotheses in Theorem \ref{thm:package-Pfw-exp-peak} for \(W_{n'}(x)=\frac{(Tx)^{\kappa n'}}{\Gamma(\kappa n'+1)}\). For $x>0$ and $u\ge0$, set
\[
R_u(x):=(Tx)^\kappa\frac{\Gamma(\kappa u+1)}{\Gamma(\kappa u+\kappa+1)},
\quad
h_x(u):=\log R_u(x).
\]
By the standard gamma-ratio asymptotic, as $u\to+\infty$, we have 
\begin{equation}\label{eq:gamma-ratio}
\frac{\Gamma(u+\kappa)}{\Gamma(u)}
=
u^\kappa\bigl(1+O(u^{-1})\bigr),
\end{equation}
implying
\[
R_u(x)=\left(\frac{Tx}{\kappa u}\right)^\kappa \left(1+O\!\left(\frac1u\right)\right),\quad h_x(u)=\kappa\log\frac{Tx}{\kappa u}+O\!\left(\frac1u\right),
\]
Set $N(x):=\Bigl\lfloor \frac{Tx}{\kappa}\Bigr\rfloor$.
Choose \(\eta\in(0,1)\) such that \(R_u(x)\ge \left(\eta\frac{Tx}{\kappa u}\right)^\kappa\) for all \(u\ge 1\) and $x>0$. As a consequence,
\begin{align}\label{eq-Wincreasewrtn}
1\le n'\le \eta\frac{Tx}{\kappa}\Rightarrow\frac{W_{n'+1}(x)}{W_{n'}(x)} = R_{n'}(x)>1.
\end{align}
Let $\psi(y):=\frac{\Gamma'(y)}{\Gamma(y)}$. We have \(h_x'(u)=\kappa\psi(\kappa u+1)-\kappa\psi(\kappa u+\kappa+1).\) Since \(\psi'(y)=y^{-1}+O(y^{-2})\) as $y\to+\infty$, it follows that there exist \(c>0\) such that for \(x\) large, when $\frac{\eta Tx}{3\kappa}\le u\le\frac{3Tx}{2\kappa}$, we have
\(
h_x'(u)\le -\frac{c}{Tx}.
\)
Hence for large \(x\), when $0\le j\le \frac{Tx}{2\kappa}$,  we have
\begin{align*}
\log\frac{W_{N(x)+j}(x)}{W_{N(x)}(x)}
=&
\sum_{m=0}^{j-1}h_x\!\bigl(N(x)+m\bigr)
\le
-\frac{c\,j(j-1)}{2Tx}+jh_x\!\bigl(N(x)\bigr)\\
=&
-\frac{c\,j^2}{2Tx}+j\left(h_x\!\bigl(N(x)\bigr)+\frac{c}{2Tx}\right)
\end{align*}
and when $0\le i\le (1-\frac{\eta}{2})\frac{Tx}{\kappa}$, we have
\[
\log\frac{W_{N(x)-i}(x)}{W_{N(x)}(x)}
\le
-\frac{c\,i(i+1)}{2Tx}-ih_x\!\bigl(N(x)\bigr)
\le
-\frac{c\,i^2}{2Tx}-ih_x\!\bigl(N(x)\bigr).
\]
When $x$ is large and \(\left(1-\frac{\eta}{2}\right)\frac{Tx}{\kappa}\le i\le N(x)\), set \(i_0:=\lfloor \left(1-\frac{\eta}{2}\right)\frac{Tx}{\kappa}\rfloor\), and from \((1-\eta)\,i\le i_0\le i\) and \eqref{eq-Wincreasewrtn}, we obtain
\[
\log\frac{W_{N(x)-i}(x)}{W_{N(x)}(x)}
\le
\log\frac{W_{N(x)-i_0}(x)}{W_{N(x)}(x)}
\le
-c\,\frac{i_0^2}{2Tx}-i_0h_x\!\bigl(N(x)\bigr)
\le
-(1-\eta)^2\cdot c\,\frac{i^2}{2Tx}-i_0h_x\!\bigl(N(x)\bigr).
\]
Note that the terms \(O\!\left(\frac1u\right)\) in $h_x(u)$ are independent of \(x\). Therefore we get
\[
h_x\!\bigl(N(x)\bigr)=O\!\left(\frac1x\right),\quad x\to+\infty,
\]
and therefore
\[
N(x)\cdot \left|h_x\!\bigl(N(x)\bigr)\right|<\log K \text{ and } N(x)\cdot \left|h_x\!\bigl(N(x)\bigr)+\frac{c}{2Tx}\right|<\log K 
\]
for some $K>0$ and all large $x$.
Finally, as $x\to+\infty$, if \(j\ge \frac{Tx}{2\kappa}\), then \(N(x)+j\ge \frac{3Tx}{2\kappa}+O(1)\), and we have
\[
h_x\!\bigl(N(x)+j\bigr)
=
\kappa\log\frac{Tx}{\kappa(N(x)+j)}+O\!\left(\frac1x\right)
\le
-\frac{\kappa}{2}\log\frac32.
\]
Therefore there exists \(q\in(0,1)\) such that when $j\ge \frac{Tx}{2\kappa}$, we have
\[
\frac{W_{N(x)+j+1}(x)}{W_{N(x)+j}(x)}
=
\exp\!\bigl(h_x(N(x)+j)\bigr)
\le q.
\]
Thus the hypotheses of Theorem \ref{thm:package-Pfw-exp-peak} are satisfied. This proves (i).

For (ii), note that $H_n=0$ when $\kappa\nmid n$ and $H_{\kappa n'}=\frac{A_{n'}T^{\kappa n'}}{(\kappa n')!}$.  Then (i) gives, as $x\to+\infty$
\[
\frac{\sum_{n'=0}^{\lfloor\frac{Tx}{\kappa}(1-(Tx)^{-\nu})\rfloor}H_{\kappa n'}x^{\kappa n'}}{H(x)}=O(e^{-\alpha x^{1-2\nu}}),\quad\frac{\sum_{n'\ge\lfloor\frac{Tx}{\kappa}(1+(Tx)^{-\nu})\rfloor}H_{\kappa n'}x^{\kappa n'}}{H(x)}=O(e^{-\alpha x^{1-2\nu}}).
\]
The proof is finished by the following Lemma \ref{lemma-mutiplicityrelation}.
\end{proof}

\begin{lemma}\label{lemma-mutiplicityrelation}
Let $H(x)=\sum_{n\ge0}H_nx^n$ be absolutely monotonic, and $\kappa\in\bZ_{>0}$. Suppose that $H_n=0$ when $\kappa\nmid n$. Assume that there exists positive numbers $C_1,C_2,\nu,\alpha,\beta$, such that, as $x\to+\infty$,
\[\frac{\sum_{n'=0}^{\lfloor C_1 x(1-C_2x^{-\nu})\rfloor}H_{\kappa n'}x^{\kappa n'}}{H(x)}=O\left(e^{-\alpha x^{\beta}}\right),\quad\frac{\sum_{n'\ge\lfloor C_1 x(1+C_2x^{-\nu})\rfloor}H_{\kappa n'}x^{\kappa n'}}{H(x)}=O\left(e^{-\alpha x^{\beta}}\right).\]
Then the summands of $H(x)=\sum_{n\ge0}H_nx^n$ concentrate exponentially near $n\approx \kappa C_1x$ with window term $C_2x^{-\nu}$ and parameters $(\alpha,\beta)$.
\end{lemma}
\begin{proof}
     We need to show that, as $x\to+\infty$,
    \[
    \frac{\sum_{n=0}^{\lfloor \kappa C_1x(1-C_2x^{-\nu})\rfloor}H_nx^n}{H(x)}=O\left(e^{-\alpha x^{\beta}}\right),\quad \frac{\sum_{n\ge\lfloor\kappa C_1 x(1+C_2x^{-\nu})\rfloor}H_{n}x^{n}}{H(x)}=O\left(e^{-\alpha x^{\beta}}\right).
    \]
    For the first equality, it suffices to show that, when $x>>0$, \[\sum_{n'=0}^{\lfloor C_1x(1-C_2x^{-\nu})\rfloor}H_{\kappa n'}x^{\kappa n'}\ge\sum_{n=0}^{\lfloor\kappa C_1x(1-C_2x^{-\nu})\rfloor}H_{n}x^{n}.\]To this end, using $H_n=0$ when $\kappa\nmid n$, we only need to show that,  when $x>>0$, \[
    \lfloor\kappa C_1x(1-C_2x^{-\nu})\rfloor<\kappa\lfloor C_1x(1-C_2x^{-\nu})\rfloor+\kappa,\]The inequality follows from \(C_1x(1-C_2x^{-\nu})<\lfloor C_1x(1-C_2x^{-\nu})\rfloor+1\). This proves the first equality.

For the second equality, it sufficies to show that, when $x>>0$, 
\[\sum_{n'\ge\lfloor C_1x(1+C_2x^{-\nu})\rfloor}H_{\kappa n'}x^{\kappa n'}\ge\sum_{n\ge\lfloor\kappa C_1x(1+C_2x^{-\nu})\rfloor}H_{n}x^{n}.\]To this end, using $H_n=0$ when $\kappa\nmid n$, we only need to show that,  
\[\lfloor\kappa C_1x(1+C_2x^{-\nu})\rfloor\ge\kappa\lfloor C_1x(1+C_2x^{-\nu})\rfloor.\]
This inequality follows from $C_1x(1+C_2x^{-\nu})\ge\lfloor C_1x(1+C_2x^{-\nu})\rfloor$. This proves the lemma.
\end{proof}

\begin{theorem}\label{thm-applicationmultigamma}
Let \(p\ge0\), \(q>0\) be integers, and let
\[
\alpha_1,\dots,\alpha_p,\beta_1,\dots,\beta_q,
\quad
a_1,\dots,a_p,b_1,\dots,b_q,
\quad
T\]
be positive numbers such that \(\kappa:=\sum_{s=1}^q\beta_s-\sum_{r=1}^p\alpha_r\in\bZ_{>0}\). Define \(C:=\prod_{r=1}^p\alpha_r^{\alpha_r}\prod_{s=1}^q\beta_s^{-\beta_s}
\).
Let \(\{a_{n'}\}_{n'\ge0}\) be a sequence in $\bR_{>0}$ such that \(\frac{a_{n'+1}}{a_{n'}}=1+O\!\left(\frac{1}{n'}\right)\) as \(n'\to\infty\), and set
\begin{equation}\label{eq:GX-general-series}
H(x):=
\sum_{n'=0}^\infty
a_{n'}
\frac{\prod_{r=1}^p\Gamma(\alpha_r n'+a_r)}
     {\prod_{s=1}^q\Gamma(\beta_s n'+b_s)}
\,(Tx)^{\kappa n'}.
\end{equation}
Then for any $\nu\in(0,\frac{1}{2})$,
the summands of \(H(x)=\sum_{n\ge0}H_nx^n\) concentrate exponentially near \(n\approx \kappa C^{\frac{1}{\kappa}}Tx\), with window terms \((\kappa C^{\frac{1}{\kappa}}T)^{-\nu}x^{-\nu}\), and parameters $(\alpha,1-2\nu)$ for some $\alpha>0$.
\end{theorem}

\begin{proof}
Note that $H_n=0$ when $\kappa\nmid n$ and $H_{\kappa n'}=a_{n'}
\frac{\prod_{r=1}^p\Gamma(\alpha_r n'+a_r)}
     {\prod_{s=1}^q\Gamma(\beta_s n'+b_s)}T^{\kappa n'}$. We set 
\[
A_{n'}:=a_{n'}
\frac{\prod_{r=1}^p\Gamma(\alpha_r n'+a_r)}{\prod_{s=1}^q\Gamma(\beta_s n'+b_s)}\frac{\Gamma(\kappa n'+1)}{(\kappa^\kappa C)^{n'}},\quad W_{n'}(x):=\frac{(\kappa C^{1\over\kappa}Tx)^{\kappa n'}}{\Gamma(\kappa n'+1)}.
\]
Then \(H(x)=\sum_{n'\ge0}H_{\kappa n'}x^{\kappa n'}=\sum_{n'\ge0}A_{n'} W_{n'}(x)\). Now
\[
\frac{A_{n'+1}}{A_{n'}}
=
\frac{a_{n'+1}}{a_{n'}}\,
\prod_{r=1}^p
\frac{\Gamma(\alpha_r n'+\alpha_r+a_r)}{\Gamma(\alpha_r n'+a_r)}
\prod_{s=1}^q
\frac{\Gamma(\beta_s n'+b_s)}{\Gamma(\beta_s n'+\beta_s+b_s)}\frac{\Gamma(\kappa n'+\kappa+1)}{\Gamma(\kappa n'+1)}(\kappa^\kappa C)^{-1}.
\]
Using the gamma-ratio asymptotic \eqref{eq:gamma-ratio}, as $n'\to\infty$, we have
\[
\frac{A_{n'+1}}{A_{n'}}
=
\frac{a_{n'+1}}{a_{n'}}\,
(\kappa^\kappa C)^{-1}
(\kappa n')^\kappa
\prod_{r=1}^p(\alpha_r n')^{\alpha_r}
\prod_{s=1}^q(\beta_s n')^{-\beta_s}
\left(1+O\!\left(\frac{1}{n'}\right)\right)
=
1+O\!\left(\frac{1}{n'}\right).
\]
Applying Theorem  \ref{thm:exp-peak-and-asymptotic} to \(\sum_{n'\ge0}A_{n'} W_{n'}(x)\), we get the required result.
\end{proof}

\section{Application to quantum periods}

In this section, after a brief review on quantum periods for Fano manifolds, we apply Theorem \ref{thm-applicationmultigamma} to prove Theorem \ref{thm-mainthm}, which is the main result of this article. We switch the variable from $x$ to $t$, in accordance with the standard notation for quantum periods.

For a Fano manifold $X$, let $\mathrm{Eff}(X)\subset H_2(X,\bZ)$ be the semi-group generated by effective curve classes. For $\bfd\in\Eff(X)$, let $\overline{M}_{0,k}(X,\bfd)$ denote the moduli space of $k$-pointed genus-zero  stable maps of class $\bfd$. For non-negative integers $a_1,\dots,a_k$ and cohomology classes $\gamma_1,\dots,\gamma_k\in H^\mathrm{ev}(X)$, the associated descendant invariant is
\[
\left<\tau_{a_1}(\gamma_1),\dots,\tau_{a_k}(\gamma_k)\right>_\bfd^X:=\int_{[\overline{M}_{0,k}(X,\bfd)]^{\mathrm{vir}}}\prod_{i=1}^k\left(c_1(\mathcal{L}_i)^{a_i}\cup\mathrm{ev}_i^*(\gamma_i)\right),
\]
where $\mathcal{L}_i$ is the universal cotangent line bundle at the $i$-th marked point, $\overline{M}_{0,k}(X,\bfd)\xrightarrow{\mathrm{ev}_i}X$ is the evaluation map at the $i$-th marked point, and $[\overline{M}_{0,k}(X,\bfd)]^{\mathrm{vir}}$ is the \emph{virtual fundamental class}.

Let $\{T_i\}$ and $\{T^i\}$ be bases of $H^\ev(X)$ dual to each other w.r.t. Poincar\'e pairing. Givental's $J$-function is defined by
\[
J_X(t):=\rme^{c_1(X)\log t}\left\{\mathbf{1}+\sum_{\bfd\in\mathrm{Eff}(X)}\sum_i\sum_{a\ge0}\left<\tau_a(T_i)\right>_\bfd^X T^it^{c_1(X)\cdot\bfd}\right\},
\]
where $\mathbf{1}\in H^0(X)$ is the identity class. Note that $J_X(t)$ is an $H^\ev(X)$-valued function, and the \emph{quantum period} $G_X(t)$ is the $H^0$-part of $J_X(t)$. We can write $G_X(t)=1+\sum_{n\ge1}G_nt^n$, with 
\[
G_n:=\sum_{\bfd\in\Eff(X):c_1(\bfd)=n}\left<\tau_{n-2}([\mathrm{pt}])\right>_\bfd^X,\quad n\ge2,
\]
and $G_1=0$, where $[\mathrm{pt}]\in H^\mathrm{top}(X)$ is the Poincar\'e dual of the class of a point. The \emph{regularized quantum period} $\widehat{G}_X(t)$ is the power series defined by 
\[
\widehat{G}_X(t):=1+\sum_{n\ge1}n!G_nt^n.
\]
There are some conjectures concerning the sequence $\{G_n\}$. For example, it is expected that $n!G_n\in\bZ_{\ge0}$ for all $n$. This holds for all known numerical examples, and is proved under some conditions (see \cite[Theorem 1.2]{Ma19} and \cite[Corollary 1.3]{Jo25}). Another conjecture concerns its asymptotic behavior. It is known that the regularized quantum period $\widehat{G}_X(t)$ has non-zero radius of convergence. The \emph{$A$-model conifold value} $T_{A,\mathrm{con}}$ of $X$ is the inverse of the radius of convergence of $\widehat{G}_X(t)$ \cite{Galkin2024RevisitingGammaI}, i.e.
\[
T_{A,\mathrm{con}}:=\varlimsup_{n\to\infty}|n!G_n|^{1\over n}.
\]
It is a non-negative real number, and we anticipate that it is always positive. 

Now we  state and prove the main result of this article.
\begin{theorem}\label{thm-mainthm}
Suppose that $X$ is a Fano manifold admitting a convenient weak LG model with non-negative coefficients. Then the summands of $G_X(t)=\sum_{n\ge0}G_nt^n$ concentrate exponentially near $n\approx T_{A,\mathrm{con}}t$, with window term $T^{-\nu}_{A,\mathrm{con}}t^{-\nu}$, and parameters $(\alpha,1-2\nu)$ for any $\nu\in(0,\frac{1}{2})$ and some $\alpha>0$.
\end{theorem}
\begin{proof}
Let $f(\mathrm{x})\in\bR_{\ge0}[x^\pm_1,\dots,x_m^{\pm}]\setminus\{0\}$ be a convenient weak LG model of $X$. Recall that ``weak LG of $X$'' means
\[
G_X(t)=1+\sum_{n>0}G_nt^n,\text{ with }G_n=\frac{\mathrm{Cst}(f^n)}{n!},
\]
where $\mathrm{Cst}(f^n)$ is the constant term of the Laurent polynomial $f(\mathrm{x})^n$,
and ``convenient" means the Newton polytope of $f(\mathrm{x})$ contains the origin in its interior.  From \cite{Ga14}, the convenient condition implies that \(f|_{(\bR_{>0})^m}\) has a unique critical point \(\mathbf{x}_{\mathrm{con}}\in(\bR_{>0})^m\), called the \emph{conifold point}, and the corresponding critical value \(T_{\mathrm{con}} := f(\mathbf{x}_{\mathrm{con}})>0\) is called the \emph{conifold value}; moreover, we have \(\mathrm{Hess}(f)(\mathbf{x}_{\mathrm{con}}):=\det(\partial_{\log x_i}\partial_{\log x_j}f(\mathbf{x}_{\mathrm{con}}))>0\). To prove Theorem \ref{thm-mainthm}, we only need to show that the summand of $G_X(t)$ concentrate exponentially near $n\approx T_{\mathrm{con}}t$, since Theorem \ref{lem:0} (iii) gives $T_{\mathrm{con}}=\varlimsup (n!G_n)^{1/n}=T_{A,\mathrm{con}}$.

From \cite[Lemma 5.8]{Galkin2024RevisitingGammaI}, there exists a unique
\(r\in\bZ_{>0}\) such that \(\mathrm{Cst}(f^n)=0\) if \(r\nmid n\), and
\(\mathrm{Cst}(f^{rn'})>0\) for all sufficiently large \(n'\). We call \(r\)
the \emph{index} of \(f\). Put
\[
g(\mathbf{x}) := f(\mathbf{x})^r,\qquad T_g:=T_{\mathrm{con}}^r .
\]
Then \(g\) is again convenient with non-negative coefficients and has index one, \(\mathbf{x}_{\mathrm{con}}\) is the conifold point of \(g\), and \(T_g\) is the conifold value of $g$.

Let us recall Galkin's random walk interpretation of $\mathrm{Cst}(g^{n'})$ \cite{Galkin2012SplitNotes}. For $\mathbf{x}=(x_1,\dots,x_m)\in(\bC^\times)^m$ and $\ell=(l_1,\dots,l_m)\in\bZ^m$, set $\mathbf{x}^\ell:=\prod_{i=1}^mx_i^{l_i}$. Then we can write \(g(\mathbf{x})=\sum_{\ell\in\bZ^m} c_\ell \mathbf{x}^{\ell}\) with \(c_\ell\ge 0\), so that only finitely many \(c_\ell\)'s are non-zero. Define \(p_\ell:=\frac{c_\ell\mathbf{x}_{\mathrm{con}}^\ell}{T_g}\). Then \(\sum_{\ell} p_\ell=1\) since $T_g=g(\mathbf{x}_{\mathrm{con}})$. Now let $\mathbb{L}\subset\bZ^m$ be the lattice generated by $\{\ell\in\bZ^m:c_{\ell}\ne0\}$, and let $(p_\ell)_{\ell\in L}$ be the step distribution of a random walk on $\mathbb{L}$ starting at the origin. Then the lattice $\mathbb{L}$ has \emph{rank $m$} since $g$ is convenient. This step distribution has \emph{finite moments of all orders} (since only finitely many $p_\ell$'s are non-zero) and \emph{zero mean}, i.e., \(\sum_\ell p_\ell\,\ell=0\), since \(\mathbf{x}_{\mathrm{con}}\) is the critical point of \(g|_{(\bR_{>0})^m}\). The random walk is \emph{irreducible} by the definition of $\mathbb{L}$. Moreover, the \(n'\)-step return probability to the origin is 
\[
\sum_{\ell_1+\cdots+\ell_{n'}=0}p_{\ell_1}\cdots p_{\ell_{n'}}=\sum_{\ell_1+\cdots+\ell_{n'}=0}\frac{c_{\ell_1}\cdots c_{\ell_{n'}}}{T_{g}^{n'}}=\frac{\mathrm{Cst}(g^{n'})}{T_g^{n'}}.
\]
Since \(g\) has index one, it follows that this probability is non-zero for all sufficiently large $n'$; hence the random walk is \emph{aperiodic}. Applying the Local Central Limit Theorem \cite[Theorem 2.3.5]{LL10} to the $n'$-step return probability and noting \cite[(2.2)]{LL10}, we obtain
\begin{equation*}\label{eq:randomwalk}
\frac{\mathrm{Cst}(g^{n'})}{T_g^{\,n'}}
=
\frac{1}{(n')^{\frac{m}{2}}}
\left(c+O\!\left(\frac{1}{\sqrt{n'}}\right)\right),
\quad n'\to+\infty,
\end{equation*}
for some constant \(c>0\). Using $g=f^r$, this implies $G_n=0$ when $r\nmid n$ and
\begin{equation*}\label{eq:gn-cst}
G_{rn'}
=
\frac{(n')^{-\frac{m}{2}}T_{\mathrm{con}}^{rn'}}{\Gamma(rn'+1)}\left(c+O\left(\frac{1}{\sqrt{n'}}\right)\right),
\quad n'\to+\infty .
\end{equation*}

Set $H_n=0$ when $r\nmid n$, $H_{rn'}=\frac{n'^{-\frac{m}{2}}T_{\mathrm{con}}^{rn'}}{\Gamma(rn'+1)}$ for $n'\ge1$, and $H_0=1$. For $H(t):=\sum_{n\ge0}H_nt^n$, since \[H(t)=\sum_{n'\ge0}s_{n'}\frac{(T_{\mathrm{con}}t)^{rn'}}{\Gamma(rn'+1)},\quad s_0:=1, \quad s_{n'}:=n'^{-\frac{m}{2}},n'\ge1.\]
it follows from Theorem \ref{thm-applicationmultigamma} that the summands of $H(t)=\sum_{n\ge0}H_nt^n$ concentrate exponentially near \(n\approx T_{\mathrm{con}}t\), with window term $T_{\mathrm{con}}^{-\nu}t^{-\nu}$ and parameters $(\alpha,1-2\nu)$ for some $\alpha>0$.

Now we have $G_X(t)=\sum_{n\ge0}c_nH_nt^n$, where $c_n=1$ when $r\nmid n$ and $c_{rn'}=c + O\!\left(\frac{1}{\sqrt{n'}}\right)$ as $n'\to\infty$. Note that the exponential concentration property is preserved when finitely many terms are replaced, with location polynomial unchanged. So without loss of generality we may assume that the sequence $\{c_n\}$ has positive lower and upper bounds, say $c_-$ and $c_+$. Set $n_\pm(t):= \lfloor T_{\mathrm{con}}t(1\pm T_{\mathrm{con}}^{-\nu}t^{-\nu})\rfloor$, and we have
\(
\frac{|\sum_{n\le n_-(t)} c_n H_nt^n|}{G_X(t)}
\le \frac{c_+}{c_-}\cdot
\frac{\sum_{n\le n_-(t)} H_nt^n}{H(t)}
\)
and
\(\frac{|\sum_{n\ge n_+(t)} c_n H_nt^n|}{G_X(t)}
\le \frac{c_+}{c_-}\cdot
\frac{\sum_{n\ge n_+(t)} H_nt^n}{H(t)}\). This implies that the summands of $G_X(t)$ concentrate exponentially near $n\approx T_{\mathrm{con}}t$. 

Now the proof of Theorem \ref{thm-mainthm} is finished.
\end{proof}

\vspace{0.5cm}

\printbibliography

\end{document}